\documentclass{article}

\newcommand{\negeer}[1]{}
\def\notintp{\mathrel{/\kern-0.6em\rhd}} 
\usepackage{amssymb}
\usepackage{amsfonts}
\usepackage{latexsym}
\usepackage{epsfig}
\usepackage{psfig}
\usepackage{xspace}
\usepackage{color}
\usepackage{amstext}
\usepackage{rotating}
\usepackage{amsmath}
\usepackage{amsthm}
\usepackage{import}
\usepackage{proof}	
\usepackage[english,dutch]{babel}
\usepackage{makeidx}

\newcommand{\pa}{\ensuremath{{\mathrm{PA}}}\xspace}

\newcommand{\il}{{\ensuremath{\textup{\textbf{IL}}}}\xspace}
\newcommand{\extil}[1]{\ensuremath{\textup{\textbf{IL}}{\sf\ensuremath{#1}}}\xspace}
\newcommand{\gl}{{\ensuremath{\textup{\textbf{GL}}}}\xspace}

\newcommand{\intl}[1]{{\ensuremath {\textup{\textbf{IL}}}({\rm #1})}}

\newcommand{\ilm}{\extil{M}}
\newcommand{\ilp}{\extil{P}}

\newcommand{\ilx}{\extil{X}}

\newcommand{\formal}[1]{\ensuremath{{\sf {#1}}\xspace}}

\newcommand{\prop}{\formal{Prop}\xspace}
\newcommand{\mcs}{\ensuremath{\textup{MCS}}\xspace}
\newcommand{\eqbydef}{:=}

\newcommand{\formil}{\ensuremath{{\sf Form}_{\il}} \xspace}

\newcommand{\dis}{\bigvee \hspace{-2.5mm} \bigvee}

\newcommand{\depth}[1]{{\ensuremath{{\sf depth}(#1)}}\xspace}

\newcommand{\sucs}{\prec}
\newcommand{\crit}[1]{\sucs_{#1}}
\newcommand{\boxin}{\subseteq_{\Box}}
\newcommand{\crone}[2]{{\ensuremath{\mathcal{C}^{#1}_{#2}}}\xspace}
\newcommand{\geone}[2]{{\ensuremath{\mathcal{G}^{#1}_{#2}}}\xspace}
\newcommand{\adset}[1]{{\ensuremath{\mathcal{#1}}\xspace}}
\newcommand{\ext}{\subseteq}

 \theoremstyle{plain}
 \newtheorem{theorem}{Theorem}[section]
 \newtheorem{lemma}[theorem]{Lemma}

 \newtheorem{corollary}[theorem]{Corollary}

 \theoremstyle{definition}
 \newtheorem{definition}[theorem]{Definition}
 
 \newtheorem{claim}{Claim}

\title{
Self Provers and $\Sigma_1$ Sentences}
\author{E. Goris\\ and\\ J. J. Joosten}
\date{2012}

\begin{document}

\maketitle

\begin{abstract}
This paper is the second in a series of three papers. All three papers deal with interpretability logics and related matters. In the first paper a construction method was exposed to obtain models of these logics. Using this method, we obtained some completeness results, some already known, and some new. 

In this paper, we will set the construction method to work to obtain more results. First, the modal completeness of the logic \ilm is proved using the construction method. This is not a new result, but by using our new proof we can obtain new results. Among these new results are some admissible rules for \ilm and \gl. 

Moreover, the new proof will be used to classify all the essentially $\Delta_1$ and also all the essentially $\Sigma_1$ formulas of \ilm. Closely related to essentially $\Sigma_1$ sentences are the so-called \emph{self provers}. A self-prover is a formula $\varphi$ which implies its own provability, that is $\varphi \to \Box \varphi$. Each formula $\varphi$ will generate a self prover $\varphi \wedge \Box \varphi$. We will use the construction method to characterize those sentences of \gl that generate a self prover that is trivial in the sense that it is $\Sigma_1$.

\end{abstract}

%





\section{Introduction}

Mathematical interpretations occur everywhere in (meta) mathematical practice. Interpretability logics study structural behavior of interpretations. One such logic, the logic \ilm, describes the structural behavior of interpretations over theories like Peano Arithmetic. In this paper we shall first study this logic \ilm and then use our findings to derive new results mainly related to $\Sigma_1$ sentences of theories like Peano Arithmetic.

This paper is the second in a series of three. For more background on interpretations and their corresponding logics we refer to the first part of this paper \cite{jogo:mm08}. Also, all definitions used in this paper occur with some motivation and background in \cite{jogo:mm08}. For completeness, self-containenedness and for readability we shall include a short recap in this paper of those technicalities that were introduced in \cite{jogo:mm08} and that are central to this paper.

\section{A concise recap: central notions of this paper}
In this paper we shall heavily resort to some rather technical results obtained in \cite{jogo:mm08}. In particular, certain parts of proofs in \cite{jogo:mm08} shall be re-used here. In this section, we shall state those parts of that paper which are necessary for results further on.

\subsection{Interpretability logics}

The modal sentences in this paper are mostly in the language of interpretability which is defined as follows. 
\[
\formil \eqbydef \bot \mid \prop \mid (\formil \rightarrow \formil) \mid (\Box \formil)
\mid (\formil \rhd \formil) 
\]

Here \prop is a countable set of propositional 
variables $p,q,r,s,t,p_0,p_1,\ldots$.
We employ the usual definitions of the logical operators
$\neg , \vee , \wedge$ and $\leftrightarrow$. Also shall we write 
$\Diamond \varphi$ for $\neg \Box \neg \varphi$. 
Formulas that start with a $\Box$ are called box-formulas or
$\Box$-formulas. Likewise we talk of $\Diamond$-formulas.

For standard reading conventions on bracketing please refer to \cite{jogo:mm08}. 

The basic interpretability logic is called \il and is captured in the following definition.

\begin{definition}\label{defi:il}
The logic \il is the smallest set of formulas being closed under
the rules of Necessitation and of Modus Ponens, that contains
all tautological formulas and all instantiations of the following
axiom schemata.

\begin{enumerate}
\item[${\sf L1}$]\label{ilax:l1}
        $\Box(A\rightarrow B)\rightarrow(\Box A\rightarrow\Box B)$
\item[${\sf L2}$]\label{ilax:l2}
        $\Box A\rightarrow \Box\Box A$
\item[${\sf L3}$]\label{ilax:l3}
        $\Box(\Box A\rightarrow A)\rightarrow\Box A$
\item[${\sf J1}$]\label{ilax:j1}
        $\Box(A\rightarrow B)\rightarrow A\rhd B$
\item[${\sf J2}$]\label{ilax:j2}
        $(A\rhd B)\wedge (B\rhd C)\rightarrow A\rhd C$
\item[${\sf J3}$]\label{ilax:j3}
        $(A\rhd C)\wedge (B\rhd C)\rightarrow A\vee B\rhd C$
\item[${\sf J4}$]\label{ilax:j4}
        $A\rhd B\rightarrow(\Diamond A\rightarrow \Diamond B)$
\item[${\sf J5}$]\label{ilax:j5}
        $\Diamond A\rhd A$
\end{enumerate}
\end{definition}
We will write $\il \vdash \varphi$ for $\varphi \in \il$. If $\sf X$ is a set of axiom schemata we will denote by \extil{X} the
logic that arises by adding the axiom schemata in $\sf X$ to \il. G\"odel L\"ob's logic \gl is obtained from \il by omitting all the $\sf J$ axioms and not allowing the $\rhd$ modality in the language.

The standard semantics for interpretability logics are given by the following definitions.

\begin{definition}\label{defi:frames}
An \il-frame is a triple $\langle W,R,S \rangle$. Here $W$ is a non-empty countable
universe,
$R$ is a binary relation on $W$ and $S$ is a set of binary relations on $W$, indexed 
by elements of $W$.
The $R$ and $S$ satisfy the following requirements.
\begin{enumerate}
\item
$R$ is conversely well-founded\footnote{A relation $R$ on $W$ is called 
conversely well-founded if every non-empty subset of $W$ has an $R$-maximal element.}

\item
$xRy \ \& \ yRz \rightarrow xRz$

\item
$yS_xz \rightarrow xRy \ \& \ xRz$

\item
$xRy \rightarrow yS_x y$

\item
$xRyRz \rightarrow yS_x z$ \label{defi:point:inclusion}

\item
$uS_x v S_x w \rightarrow u S_x w$

\end{enumerate}
\end{definition}
\il-frames are sometimes also called Veltman frames.
We will on occasion speak of $R$ or $S_x$ transitions instead of relations.
If we write $ySz$, we shall mean that $yS_xz$ for some $x$. 
$W$ is sometimes called the universe, or domain, of the frame and its elements
are referred to as worlds or nodes. With $x{\upharpoonright}$ we shall denote
the set $\{ y \in W \mid x Ry \}$.
We will often represent $S$ by a ternary relation in the canonical way, writing
$\langle x,y,z\rangle$ for $yS_xz$.

\begin{definition}
An \il-model is a quadruple 
$\langle W, R , S, \Vdash \rangle$. Here $\langle W, R , S, \rangle$ is 
an \il-frame and $\Vdash$ is a subset of $W \times \prop$. We write
$w \Vdash p$ for $\langle w,p\rangle \in \ \Vdash$.
As usual, $\Vdash$ is extended to a subset $\widetilde{\Vdash}$ of 
$W \times \formil$ by demanding the following.
\begin{itemize}

\item
$w \widetilde{\Vdash} p$ iff $w \Vdash p$ for $p \in \prop$

\item
$w \not \widetilde{\Vdash} \bot$

\item
$w \widetilde{\Vdash} A \rightarrow B$ iff $w \not \widetilde{\Vdash} A$ or 
$w \widetilde{\Vdash} B$

\item
$w \widetilde{\Vdash} \Box A$ iff 
$\forall v \ (wRv \Rightarrow v \widetilde{\Vdash} A)$

\item
$w \widetilde{\Vdash} A \rhd B$ iff 
$\forall u \ (w R u \wedge u\widetilde{\Vdash} A
\Rightarrow \exists v \ (u S_w v  \widetilde{\Vdash} B))$

\end{itemize}
\end{definition}

Note that $\widetilde{\Vdash}$ is completely determined by $\Vdash$.
Thus we will denote $\widetilde{\Vdash}$ also by $\Vdash$. It is an easy observation that the truth of a modal formula in a particular world in the model is completely determined by the part of the model that ``can be seen" from that world. This observation is used often and therefore we explicitly restate it here. 

\begin{definition}[Generated Submodel]
Let $M = \langle W, R, S , \Vdash \rangle$ be an \il-model
and let $m\in M$. We define $m{\upharpoonright}*$ to be the
set $\{ x\in W \mid x{=}m \vee mRx \}$. 
By $M{\upharpoonright}m$ we denote the 
submodel generated by $m$ defined as follows.
\[
M{\upharpoonright}m \eqbydef \langle m{\upharpoonright}*, 
R\cap (m{\upharpoonright}*)^2 , \bigcup_{x\in m{\upharpoonright}*}
S_x \cap (m{\upharpoonright}*)^2, \Vdash \cap 
(m{\upharpoonright}* \times {\sf Prop})  \rangle
\]
\end{definition}

\begin{lemma}[Generated Submodel Lemma]\label{lemm:gesulem}
Let $M$ be an \il-model and let $m\in M$. For all 
formulas $\varphi$ and all $x \in m{\upharpoonright}*$ we have that
\[
M{\upharpoonright}m,x\Vdash \varphi \ \ \mbox{ iff }\ \ 
M,x\Vdash \varphi .
\]
\end{lemma}

In \cite{jogo:mm08} models are built by gluing sets of modal sentences together. We shall briefly recapitulate the main definitions of those sets of sentences here.

\begin{definition}
A set $\Gamma$ is \extil{X}-consistent iff $\Gamma \not \vdash_{\extil{X}} \bot$.
An \extil{X}-consistent set is maximal \extil{X}-consistent if for 
any $\varphi$, either $\varphi \in \Gamma$ or $\neg \varphi \in \Gamma$.
\end{definition}

We will often abbreviate ``maximal consistent set'' by \mcs and 
refrain from explicitly mentioning the logic \extil{X} when the
context allows us to do so.
We define three useful relations on \mcs's, the \emph{successor} relation $\sucs$, 
the \emph{$C$-critical successor} relation $\crit{C}$ and the 
\emph{Box-inclusion} relation $\boxin$.

\begin{definition}\label{defi:mcsrel}
Let $\Gamma$ and $\Delta$ denote maximal \extil{X}-consistent sets.
\begin{itemize}
\item
$\Gamma \sucs \Delta \eqbydef \Box A \in \Gamma \Rightarrow A, \Box A \in \Delta$
\item
$\Gamma \crit{C} \Delta \eqbydef A \rhd C \in \Gamma 
\Rightarrow \neg A , \Box \neg A \in \Delta$
\item
$\Gamma  \boxin \Delta \eqbydef \Box A \in \Gamma \Rightarrow \Box A \in \Delta$
\end{itemize}
\end{definition}

It is clear that $\Gamma \crit{C} \Delta \Rightarrow \Gamma \sucs \Delta$.
For, if $\Box A \in \Gamma$ then $\neg A \rhd \bot \in \Gamma$. Also 
$\bot \rhd C \in \Gamma$, whence $\neg A \rhd C \in \Gamma$. If now 
$\Gamma \crit{C} \Delta$ then $A , \Box A \in \Delta$, whence 
$\Gamma \sucs \Delta$. It is also clear that 
$\Gamma \crit{C} \Delta \sucs \Delta' \Rightarrow \Gamma \crit{C} \Delta'$.

\begin{lemma}\label{lemm:botcrit}
Let $\Gamma$ and $\Delta$ denote maximal \extil{X}-consistent sets. We have
$\Gamma \sucs \Delta$ iff $\Gamma \crit{\bot} \Delta$.
\end{lemma}

\subsection{The construction method and the Main Lemma}
The main purpose of \cite{jogo:mm08} was to provide some background in the modal theory of provability logics. Moreover, in that paper, a construction method was developed. The construction method provided a way of gluing sets of modal sentences together as to obtain models with desired properties. The ideas involved are quite similar to the definition of canonical models with the exception that the model is constructed step by step rather than defined at once and, moreover, only that part of the model that you need is constructed and nothing more.

Thus, the building blocks are maximal consistent sets of modal interpretability logics. Instead of gluing these sets together outright, we shall glue variables $u,v, \ldots$ together and label these variables by the sets. We denote the labeling by $\nu$. Thus, if we added a new element $x$, by $\nu(x)$ we refer to the corresponding set of modal sentences. Likewise, certain $R$ transitions will be labeled via $\nu$ with a single formula, for example $\nu(\langle x,y \rangle) =C$.

The following two notions are central to the construction method. As they are so central to the paper we strongly advice the reader who is novice to this field to read the motivation of these notions in Section 3 of \cite{jogo:mm08}.

\begin{definition}\label{defi:critcone}
Let $x$ be a world in some \extil{X}-labeled frame $\langle W, R, S, \nu \rangle$.
The \emph{$C$-critical cone above $x$}, we write $\crone{C}{x}$, is defined inductively as
\begin{itemize}
\item
$\nu (\langle x,y \rangle)=C \Rightarrow y \in \crone{C}{x}$
\item
$x' \in \crone{C}{x} \ \& \  x'S_x y \Rightarrow y\in \crone{C}{x}$
\item
$x' \in \crone{C}{x} \ \& \  x'R y \Rightarrow y\in \crone{C}{x}$
\end{itemize}
\end{definition}

\begin{definition}\label{defi:gencone}
Let $x$ be a world in some \extil{X}-labeled frame $\langle W, R, S, \nu \rangle$.
The \emph{generalized $C$-cone above $x$}, we write $\geone{C}{x}$, is defined inductively as
\begin{itemize}
\item
$y \in \crone{C}{x} \Rightarrow  y \in \geone{C}{x}$
\item
$x' \in \geone{C}{x}  \ \& \ x'S_wz \Rightarrow z \in \geone{C}{x}$ for arbitrary $w$
\item
$x' \in \geone{C}{x}  \ \& \ x'Ry \Rightarrow y \in \geone{C}{x}$
\end{itemize}
\end{definition}

The construction method in essence deals step by step with existential requirements -- so-called \emph{problems}-- and with universal requirements -- so-called \emph{deficiencies} -- both defined below.

\begin{definition}[Problems]\label{defi:problems}
Let \adset{D} be some set of sentences. 
A \emph{\adset{D}-problem} is a pair $\langle x , \neg (A \rhd B)\rangle$ such that 
$\neg (A \rhd B) \in \nu (x) \cap \adset{D}$ 
and for no $y \in \crone{B}{x}$ we have $A \in \nu (y)$.
\end{definition}

\begin{definition}[Deficiencies]
Let \adset{D} be some set of sentences and let 
$F = \langle W, R, S, \nu \rangle$ be an \extil{X}-labeled frame. A \emph{\adset{D}-deficiency}
is a triple $\langle x,y, C \rhd D \rangle$ with $xRy$, 
$C\rhd D \in \nu (x) \cap \adset{D}$, and $C\in \nu (y)$,
but for no $z$ with $yS_xz$ we have $D\in \nu (z)$.
\end{definition}

If the set $\adset{D}$ is clear or fixed, we will just speak about problems and deficiencies.
The labeled frames we will construct are always supposed to satisfy some minimal reasonable
requirements. We summarize these in the notion of adequacy.

\begin{definition}[Adequate frames]
A frame is called \emph{adequate} if the following conditions are satisfied.
\begin{enumerate}
\item
$xRy \Rightarrow \nu (x) \sucs \nu (y)$

\item
$A \neq B \Rightarrow \geone{A}{x} \cap \geone{B}{x} = \varnothing$

\item
$y \in \crone{A}{x} \Rightarrow \nu (x) \crit{A} \nu (y)$
\end{enumerate}
\end{definition}

We need three more technical definitions before we can re-state the Main Lemma.

\begin{definition}
Let $\mathcal{D}$ be some set of formulas and let $M$ be an interpretability model. We say that \emph{a Truth-Lemma holds on $M$ with respect to $\mathcal{D}$} if for all $x$ in $M$ we have that

\[
 \forall \, \varphi {\in} 
\adset{D} \ [ x\Vdash \varphi \mbox{ iff. } \varphi \in x]. 
\]
 
\end{definition}

\begin{definition}[Depth]
The \emph{depth} of a finite frame $F$, we will write $\depth{F}$ is the maximal length of 
sequences of the form $x_0R \ldots R x_n$. (For convenience we define 
$\max (\varnothing)=0$.)
\end{definition}

\begin{definition}[Union of Bounded Chains]
An indexed set $\{ F_i\}_{i \in \omega}$ of labeled frames is called a \emph{chain} if
for all $i$, $F_i \ext F_{i+1}$. It is called a \emph{bounded chain} if for some 
number $n$, $\depth{F_i}\leq n$ for all $i\in \omega$. The \emph{union} of a bounded 
chain $\{ F_i\}_{i \in \omega}$ of labeled frames $F_i$ is defined as follows.
\[
\cup_{i\in \omega} F_i 
\eqbydef \langle \cup_{i\in \omega} W_i,  \cup_{i\in \omega}R_i,  \cup_{i\in \omega}S_i, 
\cup_{i\in \omega}\nu_i \rangle
\]
\end{definition}

Finally the Main Lemma can be formulated.

\begin{lemma}[Main Lemma] \label{lemm:main}

Let $\extil{X}$ be an interpretability logic and let 
$\mathcal{C}$ be a (first or higher order) frame condition such that
for any \il-frame $F$ we have 
\[
F \models \mathcal{C} \Rightarrow F \models {\sf X}.
\] 
Let $\adset{D}$ be a finite set of sentences.
Let $\mathcal{I}$ be a set of so-called \emph{invariants} 
of labeled frames so that we have 
the following properties.
\begin{itemize}
\item
$F \models \mathcal{I}^{\mathcal{U}} \Rightarrow
F \models \mathcal{C}$,
where $\mathcal{I}^{\mathcal{U}}$ is that part of $\mathcal{I}$ that is closed
under bounded unions of labeled frames.

\item
$\mathcal{I}$ contains the following invariant: 
$xRy \rightarrow \exists \, A {\in} (\nu (y) \setminus \nu (x))\cap 
\{ \Box \neg D \mid D$  a subformula of some  $B \in \adset{D} \}$.
%
%
\item
For any  adequate labeled frame  $F$, satisfying all the invariants,
we have the following.

\begin{itemize}

\item
Any \adset{D}-problem of $F$ can be eliminated by extending 
$F$
in a way that conserves all invariants.

\item
Any \adset{D}-deficiency of $F$
can be eliminated by extending 
$F$
in a way that conserves all invariants.
\end{itemize}
\end{itemize}

In case such a set of invariants $\mathcal{I}$ exists, we have that
any \extil{X}-labeled adequate frame $F$ satisfying all the invariants 
can be extended to 
some labeled adequate \extil{X}-frame $\hat F$ on which a truth-lemma with 
respect to 
$\adset{D}$ holds.

Moreover, if for any finite \adset{D} that is closed under
subformulas and single negations, a corresponding set of  
invariants $\mathcal{I}$ can be found as above and such that moreover 
$\mathcal{I}$
holds on any one-point 
labeled frame, we have that $\extil{X}$ is a complete logic.
\end{lemma}

The following two lemmata indicate how problems and deficiencies can be dealt with.

\begin{lemma}\label{lemm:problems}
Let $\Gamma$ be a maximal \extil{X}-consistent set such that $\neg (A\rhd B) \in \Gamma$.
Then there exists a maximal \extil{X}-consistent set $\Delta$ such that 
$\Gamma \crit{B} \Delta \ni A, \Box \neg A$.
\end{lemma}

\begin{lemma}\label{lemm:deficiencies}
Consider $C \rhd D \in \Gamma \crit{B} \Delta \ni C$. There exists $\Delta'$
with $\Gamma \crit{B} \Delta'\ni D, \Box \neg D$. 
\end{lemma}

\subsection{Modal Completeness of \il}

In \cite{jogo:mm08} the first application of the construction method was reproving the modal completeness of \il. Large parts of this new completeness for \il can be re-used in other proofs. We mention here the ingredients of the completeness proof of \il that will be re-used in this paper.\\

\medskip

The Main Lemma will basically add bits and pieces to a model until all the necessary requirements are met. By adding a bit to a labeled frame a structure will arise that is almost a new frame but not quite yet. Those structures are called \emph{quasi frames} and are defined below.

\begin{definition}\label{defi:quasiframes}
A \emph{quasi-frame} $G$ is a quadruple $\langle W,R,S,\nu \rangle$.
Here $W$ is a non-empty set of worlds, and $R$ a binary relation on $W$.
$S$ is a set of binary relations on $W$ indexed by elements of $W$.
The $\nu$ is a labeling as defined on labeled frames. Critical cones
and generalized cones are defined just in the same way as in the
case of labeled frames.
$G$ should posess the following properties.
\begin{enumerate}
\item
$R$ is conversely well-founded

\item
$yS_xz \rightarrow xRy \ \& \ xRz$

\item
$xRy \rightarrow \nu (x) \sucs \nu (y)$

\item
$A \neq B \rightarrow \geone{A}{x} \cap \geone{B}{x} = \varnothing$

\item
$y{\in} \crone{A}{x} \rightarrow \nu (x) \crit{A} \nu (y)$

\end{enumerate}
 
\end{definition}

Once the Main Lemma is around, the main effort in the proof of the modal completeness of \il lies in showing that each quasi frame can be extended to adequate labeled frame. We restate here  this fact and hint at the main ingredients of the proof.

\begin{lemma}[\il-closure]\label{lemm:extension}
Let $G=\langle W,R,S,\nu \rangle$ be a quasi-frame. There is an adequate 
\il-frame $F$ extending $G$. That is, $F=\langle W,R',S',\nu \rangle$
with $R\subseteq R'$ and $S\subseteq S'$.
\end{lemma}

\begin{proof}
 We define an
\emph{imperfection} on a quasi-frame $F_n$ to be a tuple $\gamma$ having one of
the following forms.
\begin{itemize}
\item[$(\romannumeral 1)$]
$\gamma = \langle 0,a,b,c \rangle$ with 
$F_n \models aRbRc$ but $F_n \not \models aRc$

\item[$(\romannumeral 2)$]
$\gamma = \langle 1,a,b\rangle$ with $F_n\models aRb$ but  $F_n \not \models bS_ab$

\item[$(\romannumeral 3)$]
$\gamma = \langle 2,a,b,c,d \rangle$ 
with $F_n \models bS_acS_ad$ but not $F_n \models bS_ad$

\item[$(\romannumeral 4)$]
$\gamma = \langle 3,a,b,c\rangle$
with $F_n\models aRbRc$ but $F_n \not \models bS_ac$

\end{itemize}
Now let us start with a quasi-frame $G=\langle W,R,S,\nu \rangle$.
We will define a chain of quasi-frames. Every new element in the chain will
have at least one imperfection less than its predecessor. The union will have no
imperfections at all. It will be our required adequate \il-frame.
\end{proof}

\section{The Logic \ilm}
Let us first recall the principle $\sf M$, also called Montagna's 
principle.
\[
{\sf M}: \ \ \ A\rhd B \rightarrow A \wedge \Box C \rhd B \wedge \Box C
\]

The modal logic \ilm is of importance because it is the interpretability logic of theories like Peano Arithmetic.

\begin{theorem}[Berarducci \cite{bera:inte90}, Shavrukov \cite{shav:logi88}]\label{theo:shav}
If $T$ is an essentially reflexive theory, then $\intl{T}=\ilm$.
\end{theorem}

The modal completeness of \ilm was proved by de Jongh and Veltman in
\cite{JoVe90}.
In this section we will reprove the modal completeness of the logic 
\ilm via the Main Lemma. This is done in 
\ref{subs:ILMpreparations}
and
\ref{subs:ilmcompleteness}.
In \ref{subs:admissibleRules} the new completeness proof is used to obtain some new results on admissible rules of \ilm.

The general approach to the new completeness proof of \ilm is not much different from the completeness 
proof for \il.
The novelty consists of incorporating the \ilm frame condition,
that is,
whenever 
$yS_xzRu$ holds, we should also have $yRu$. In this case, 
adequacy imposes $\nu (y) \sucs \nu (u)$.

Thus, whenever we introduce an $S_x$ relation, when eliminating 
a deficiency, we should keep in mind that in a later stage, this 
$S_x$ can activate the \ilm frame condition. It turns out to be
sufficient to demand $\nu (y) \boxin \nu (z)$ whenever $ySz$.
Also, we should do some additional book keeping as to keep our critical 
cones fit to our purposes.

\subsection{Preparations}\label{subs:ILMpreparations}

We start by defining a frame condition for \ilm.
\begin{definition}\label{defi:ilmframe}
An \ilm-frame is a frame such that 
$yS_xzRu \rightarrow yRu$ holds on it. A(n adequate) labeled 
\ilm-frame is a labeled \ilm-frame 
on which $yS_xz\rightarrow\nu(y)\boxin \nu(z)$ 
holds. We call $yS_xzRu \rightarrow yRu$ the frame condition of \ilm.
\end{definition}

The next lemma tells us that the frame condition of \ilm, indeed 
characterizes the frames of \ilm.

\begin{lemma}\label{lemm:frameilm}
$F \models \forall x,y,u,v \; (yS_xuRv \rightarrow yRv) 
\Leftrightarrow F \models \ilm$
\end{lemma}

We will now introduce a notion of a quasi-\ilm-frame and a corresponding 
closure lemma.
In order to get an \ilm-closure lemma in analogy with Lemma \ref{lemm:extension}
we need to introduce a technicality.

\begin{definition}\label{defi:critmcone}
The $A$-critical $\mathcal{M}$-cone of $x$, we write $\mathcal{M}_x^A$,
is defined inductively as follows.
\begin{itemize}
\item
$xR^Ay \rightarrow y \in \mathcal{M}_x^A$
\item
$ y \in \mathcal{M}_x^A\ \& \  yRz \rightarrow z \in \mathcal{M}_x^A$
\item
$ y \in \mathcal{M}_x^A\ \& \  yS_xz \rightarrow z \in \mathcal{M}_x^A$
\item
$ y \in \mathcal{M}_x^A\ \& \  yS^{\sf tr} uRv \rightarrow v \in \mathcal{M}_x^A$
\end{itemize}
\end{definition}

\begin{definition}\label{defi:ilmquasiframe}
A quasi-frame is 
a quasi-\ilm-frame if\footnote{By $R^{\sf tr}$ we 
denote the transitive closure of $R$, inductively defined as
the smallest set such that $xRy \rightarrow xR^{\sf tr}y$ and
$\exists z \ (xR^{\sf tr}z \wedge zR^{\sf tr}y) \rightarrow x R^{\sf tr} y )$. 
Similarly we define $S^{\sf tr}$. The $;$ is the composition 
operator on relations. Thus, for example, $y (R^{\sf tr} ; S)z$ iff
there is a $u$ such that $yR^{\sf tr}u$ and $uSz$. Recall that $uSv$
iff $uS_xv$ for some $x$. In the literature one often also uses the 
$\circ$ notation, where $xR\circ Sy$ iff 
$\exists z \; xSzRy$. Note that $R^{\sf tr}; S^{\sf tr}$ is conversely well-founded
iff $R^{\sf tr}\circ S^{\sf tr}$ is conversely well-founded.} 
the following properties hold.
\begin{itemize}
\item
$R^{\sf tr}; S^{\sf tr}$ is conversely well-founded\footnote{In
the case of quasi-frames we did not need a second order frame condition.
We could use the second order frame condition of \il via
$yS_xz \rightarrow xRy \  \& \ xRz$. Such a trick seems not to be available 
here.}

\item
$yS_xz\rightarrow\nu(y)\boxin \nu(z)$

\item
$y \in  \mathcal{M}_x^A \Rightarrow \nu (x) \crit{A} \nu (y)$

\end{itemize}

\end{definition}
It is easy to see that 
$\crone{A}{x} \subseteq  \mathcal{M}_x^A \subseteq \geone{A}{x}$. Thus
we have that 
$A\neq B \rightarrow  \mathcal{M}_x^A \cap \mathcal{M}_x^B = \varnothing$.
Also, it is clear that if $F$ is an \ilm-frame, then 
$F\models  \mathcal{M}_x^A = \crone{A}{x}$.
Actually we have that a quasi-\ilm-frame $F$ is an \ilm-frame iff
$F\models  \mathcal{M}_x^A = \crone{A}{x}$.



\begin{lemma}[\ilm-closure]\label{lemm:ilmextension}
Let $G=\langle W,R,S,\nu \rangle$ be a quasi-\ilm-frame. There is an adequate 
\ilm-frame $F$ extending $G$. That is, $F=\langle W,R',S',\nu \rangle$
with $R\subseteq R'$ and $S\subseteq S'$.
\end{lemma}

\begin{proof}
The proof is very similar to that of Lemma \ref{lemm:extension}. As a matter
of fact, we will use large parts of the latter proof in here.
For quasi-\ilm-frames we also define the notion of an imperfection. 
An 
\emph{imperfection} on a quasi-\ilm-frame $F_n$ is a tuple $\gamma$
that is either an imperfection on the quasi-frame $F_n$, or
it is a tuple of the form
\[
\gamma = \langle 4, a,b,c,d\rangle
\mbox{ with } F_n \models bS_acRd \mbox{ but }
F_n \not \models bRd .
\]
As in the closure proof for quasi-frames, 
we define a chain of quasi-\ilm-frames. Each new frame in the chain will
have at least one imperfection less than its predecessor. 
We only have to consider the new imperfections, in which case we define
\[
F_{n+1}\eqbydef \langle W_n , R_n \cup \{  \langle b,d \rangle \}, 
S_n  , \nu_n \rangle .
\]
We now see by an easy but elaborate induction that every $F_n$ is
a quasi-\ilm-frame. Again, this boils down to checking that at
each of $(\romannumeral 1)$-$(\romannumeral 5)$, all the eight 
properties from Definition \ref{defi:ilmquasiframe} are preserved.

During the closure process, the critical cones do change. However, the 
critical $\mathcal{M}$-cones are invariant. Thus, it is useful to prove
\[
8'.\ \ F_{n+1}\models y \in \mathcal{M}^A_x \mbox{ iff } F_n\models y \in \mathcal{M}^A_x .
\]
Our induction is completely straightforward. As an example we shall see that 
$8'$ holds in Case $(\romannumeral 1)$: We have eliminated an imperfection concerning the 
transitivity of the $R$ relation and 
$F_{n+1}\eqbydef \langle W_n , R_n \cup \{  \langle a,c \rangle \}, S_n , \nu_n \rangle$.

To see that $8'$ holds, we reason as follows.
Suppose $F_{n+1}\models y \in \mathcal{M}^A_x$. Thus 
$\exists z_1,\ldots,z_l$  $(0\leq l)$ with\footnote{The union 
operator on relations can just be seen as the set-theoretical union.
Thus, for example, $y (S_x \cup R) z$ iff $yS_xz$ or
$yRz$.}
$F_{n+1}\models xR^Az_1 (S_x \cup R \cup (S^{\sf tr} ; R)) z_2, \ldots,
z_l (S_x \cup R \cup (S^{\sf tr} ; R)) y$.
We transform the sequence $z_1,\ldots,z_l$ into a sequence
$u_1,\ldots , u_m$ ($0\leq m$) in the following way.
Every occurrence of $aRc$ in $z_1 , \ldots ,z_l$ is replaced by
$aRbRc$. In case that for some $n <l$ we have
$z_n S^{\sf tr}aRc=z_{n+1}$, we replace $z_n, z_{n+1}$ by
$z_n,b,c$ and thus 
$z_n (S^{\sf tr}; R)bRc$. We leave the rest of the 
sequence $z_1,\ldots , z_l$ unchanged. 
Clearly
$F_{n}\models xR^Au_1 (S_x \cup R \cup (S^{\sf tr} ; R)) u_2, \ldots,
u_m (S_x \cup R \cup (S^{\sf tr} ; R)) y$, whence 
$F_n\models y \in \mathcal{M}^A_x$.
\medskip

We shall include one more example for Case $(\romannumeral 5)$: 
We have eliminated an imperfection 
concerning the \ilm 
frame-condition and 
$F_{n+1}\eqbydef \langle W_n , R_n \cup \{  \langle b,d \rangle \}, 
S_n  , \nu_n \rangle$.
To see the
conversely well-foundedness of $R$, we reason as follows.
Suppose for a contradiction that there is an infinite sequence such that
$F_{n+1}\models x_1Rx_2R\ldots$. We now get an infinite sequence
$y_1,y_2,\ldots$ by replacing every occurrence of $bRd$ in 
$x_1,x_2,\ldots$ by $bS_acRd$ and leaving the rest unchanged.
If there are infinitely many $S_a$-transitions in the sequence
$y_1,y_2,\ldots$ (note that there are certainly infinitely many
$R$-transitions in $y_1,y_2,\ldots$), we get a contradiction 
with our assumption that $R^{\sf tr}; S^{\sf tr}$ is
conversely well-founded on $F_n$.
In the other case we get a contradiction with the conversely 
well-foundedness of $R$ on $F_n$.
\medskip

Once we have seen that indeed, every $F_n$ is a quasi-\ilm-frame, it is not hard to see 
that $F \eqbydef \cup_{i\in \omega}F_i$ is the required adequate \ilm-frame.
To this extend we have to check a list of properties $(a.)$-$(n.)$.
The properties $(a.)$-$(l.)$ are as in the proof of Lemma \ref{lemm:extension}.

The one exception is Property $(d.)$. To see $(d.)$, the 
conversely well-foundedness of $R$, we prove by induction
on $n$ that 
$F_n\models xRy$ iff $F_0 \models x (S^{\sf tr, refl};R^{\sf tr})y$.
Thus, a hypothetical infinite sequence
$F\models x_0 R x_1 R x_2 R \ldots$
defines an infinite sequence
$F_0\models x_0 (S^{\sf tr, refl};R^{\sf tr}) x_1 
(S^{\sf tr, refl};R^{\sf tr}) x_2 \ldots$, which contradicts either the
conversely well-foundedness of $R$ or of $S^{\sf tr};R^{\sf tr}$ on $F_0$.

The only new properties in this list 
are $(m.): \ uS_xvRw \rightarrow uRw$ and 
$(n.): \ yS_xz \rightarrow \nu (y) \boxin \nu (z)$, but
they are easily seen to hold on $F$.
\end{proof}

Again do we note that the closure obtained in Lemma \ref{lemm:ilmextension}
is unique. Thus we can refer to the \ilm-closure of a quasi-\ilm-frame.
All the information about the labels can be dropped in 
Definition \ref{defi:ilmquasiframe} and  Lemma \ref{lemm:ilmextension}
to obtain a lemma about regular \ilm-frames.

\begin{corollary}\label{coro:invariantilm}
Let \adset{D} be a finite set of sentences, closed under subformulas and single negations.
Let $G=\langle W,R,S,\nu \rangle$ be a quasi-\ilm-frame on which
\[
xRy \rightarrow \exists \,  A {\in}((\nu (y)\setminus \nu (x))\cap 
\{  \Box D \mid D \in \adset{D}\} ) \ \ \ (*)
\]
holds. Property $(*)$ does also hold on the \il-closure $F$ of $G$.
\end{corollary}

\begin{proof}
The proof is as the proof of  
Corollary 5.3 from \cite{jogo:mm08}.
We only need to remark on Case $(\romannumeral 5)$: If $bS_acRd$, we have
$\nu (b) \boxin \nu (c)$. Thus, 
$A\in ((\nu (d) \setminus \nu (c)) \cap \{ \Box D \mid D \in \adset{D}\})$
implies $A \not \in \nu (b)$.
\end{proof}

The final lemma in our preparations is a lemma that is needed to 
eliminate deficiencies properly.

\begin{lemma}\label{lemm:ilmdeficiencies}
Let $\Gamma$ and $\Delta$ be maximal \ilm-consistent sets.
Consider $C \rhd D \in \Gamma \crit{B} \Delta \ni C$. There exists a 
maximal \ilm-consistent set $\Delta'$
with $\Gamma \crit{B} \Delta'\ni D, \Box \neg D$ and $\Delta \boxin \Delta'$. 
\end{lemma}

\begin{proof}
By compactness and by commutation of boxes and conjunctions, 
it is sufficient to show that for any 
formula $\Box E \in \Delta$ there is a $\Delta''$ with 
$\Gamma \crit{B} \Delta'' \ni D  \wedge \Box E \wedge \Box \neg D$. As 
$C\rhd D$ is in the maximal \ilm-consistent set $\Gamma$, also 
$C \wedge \Box E \rhd D \wedge \Box E \in \Gamma$.
Clearly $C \wedge \Box E \in \Delta$, whence, by Lemma
\ref{lemm:deficiencies} we find a $\Delta''$ with 
$\Gamma \crit{B} \Delta'' \ni D\wedge \Box E \wedge \Box (\neg D \vee \neg \Box E)$.
As $\ilm \vdash \Box E \wedge \Box (\neg D \vee \neg \Box E) \rightarrow \Box \neg D$,
we see that also $D \wedge \Box E \wedge \Box \neg D \in \Delta''$.
\end{proof}

\subsection{Completeness}\label{subs:ilmcompleteness}

\begin{theorem}\label{theo:ilmcomplete}
\ilm is a complete logic.
\end{theorem}

\begin{proof}
{\bf Frame Condition} In the case of \ilm the frame condition is 
easy and well known, as
expressed in Lemma \ref{lemm:frameilm}.

\medskip
{\bf Invariants}
Let \adset{D} be a finite set of sentences closed under subformulas and 
single negations. We define a corresponding set of invariants.
\[
\mathcal{I} \eqbydef \left \{ \begin{array}{l} xRy 
\rightarrow \exists \,  A {\in}((\nu (y)\setminus \nu (x))\cap 
\{  \Box D \mid D \in \adset{D}\} ) \\ 
uS_xvRw\rightarrow uRw
\end{array}\right.
\]
\medskip

{\bf Elimination} Thus, we consider an \ilm-labeled frame 
$F \eqbydef \langle W , R, S, \nu \rangle$ that satisfies the invariants.

\medskip
{\bf Problems} Any problem $\langle a, \neg (A \rhd B)\rangle$ of $F$ will
be eliminated in two steps.

\begin{enumerate}
\item
Using Lemma \ref{lemm:problems} we can find a \mcs $\Delta$ with 
$\nu (a) \crit{B} \Delta \ni A, \Box \neg A$. We fix some $b\notin W$
and define
\[
G':= \langle W\cup \{  b\} , R \cup \{ \langle a, b\rangle \}, S, \nu \cup
\{ \langle  b, \Delta \rangle , \langle \langle a,b \rangle, B \rangle \}\rangle.
\]
We now see that $G'$ is a quasi-\ilm-frame. Thus, we need to check the eight points
from Definitions \ref{defi:ilmquasiframe} and \ref{defi:quasiframes}.
We will comment on some of these points.

To see, for example, Point 4,
$C\neq D \rightarrow \geone{C}{x} \cap \geone{D}{x} = \varnothing$, we reason as
follows. 
First, we notice that 
$\forall \,  x,y {\in} W \ [G'\models y \in \geone{C}{x} \mbox{ iff } 
F\models y \in \geone{C}{x}]$ holds for any $C$.
Suppose $G' \models \geone{C}{x} \cap \geone{D}{x} \neq \varnothing$.
If $G' \models b \notin \geone{C}{x} \cap \geone{D}{x}$, then also
$F\models \geone{C}{x} \cap \geone{D}{x} \neq \varnothing$. As $F$ is
an \ilm-frame, it is certainly a quasi-\ilm-frame, whence $C=D$.
If now $G' \models b \in \geone{C}{x} \cap \geone{D}{x}$, necessarily
$G' \models a \in \geone{C}{x} \cap \geone{D}{x}$, whence
$F \models a \in \geone{C}{x} \cap \geone{D}{x}$ and $C=D$.

To see Requirement 8, 
$y\in \mathcal{M}^E_x \rightarrow \nu (x) \crit{E} \nu (y)$, we 
reason as follows. 
Again, we first note that 
$\forall \,  x,y {\in} W \ [G'\models y \in \mathcal{M}^C_x \mbox{ iff } 
F\models y \in \mathcal{M}^C_x]$ holds for any $C$.
We only need to consider the new element, that is,
$b\in \mathcal{M}^E_x$.
If $x=a$ and $E=B$, we get the property by choice of $\nu (b)$.

For $x\neq a$, we consider two cases. Either $a\in \mathcal{M}^E_x$
or $a\notin \mathcal{M}^E_x$. In the first case, we get by the 
fact that $F$ is a labeled \ilm-frame $\nu (x) \crit{E} \nu (a)$.
But $\nu (a) \sucs \nu (b)$, whence $\nu (x) \crit{E} \nu (b)$.
In the second necessarily for some $a'\in \mathcal{M}^E_x$
we have $a'S^{\sf tr}a$. But now $\nu (a') \boxin \nu (a)$. Clearly 
$\nu (x) \crit{E} \nu (a') \boxin \nu (a) \sucs \nu (b) \rightarrow 
\nu (x) \crit{E} \nu (b)$.

\item
With Lemma \ref{lemm:ilmextension} we extend $G'$ to an adequate labeled \ilm-frame
$G$. It is now obvious that both of the invariants hold on $G$. 
The first one holds due to Corollary \ref{coro:invariantilm}. The other
is just included in the definition of \ilm-frames.
Obviously, $\langle a, \neg (A \rhd B ) \rangle$ is not a problem any more in 
$G$.

\end{enumerate}
\medskip

{\bf Deficiencies}. Again, any deficiency $\langle a,b, C \rhd D\rangle$ in 
$F$ will be eliminated in two steps.
\begin{enumerate}
\item
We first define $B$ to be the formula such that $b\in \crone{B}{a}$. If such a
$B$ does not exist, we take $B$ to be $\bot$. Note that if such a $B$ does exist,
it must be unique by Property $4$ of Definition \ref{defi:quasiframes}.
By Lemma \ref{lemm:botcrit}, or just by the fact that $F$ is an \ilm-frame,
we have that $\nu (a) \crit{B} \nu (b)$.

By Lemma \ref{lemm:ilmdeficiencies} we can now find a 
$\Delta'$ such that
$\nu (a) \crit{B} \Delta' \ni D, \Box \neg D$ and 
$\nu (b) \boxin \Delta'$. We fix some $c \not \in W$
and define
\[
G' \eqbydef \langle W, R \cup \{ \langle a,c \rangle \}, 
S \cup \{ \langle a,b,c \rangle \}, \nu \cup \{  \langle c, \Delta' \rangle \} \rangle .
\]
To see that $G'$ is indeed a quasi-\ilm-frame, again eight properties 
should be checked. But all of these are fairly routine.

For Property 4 it is good to remark that, if $c\in \geone{A}{x}$, then
necessarily $b\in \geone{A}{x}$ or
$a \in \geone{A}{x}$.

To see Property 8, we reason as follows. We only need to consider
$c\in \mathcal{M}^A_x$. This is possible 
if $x=a$ and
$b\in \mathcal{M}^A_a$, or if for some $y \in \mathcal{M}^A_x$ we have
$yS^{\sf tr}a$, or if $a\in \mathcal{M}^A_x$. 
In the first case, we get that $b\in \mathcal{M}^A_a$,
and thus also $b\in \crone{A}{a}$ as $F$ is an \ilm-frame. Thus, by 
Property 4, we see that $A=B$. But $\Delta'$ was chosen such that
$\nu (a)\crit{B} \Delta'$. In the second case we see that 
$\nu (x) \crit{A}\nu (y) \boxin \nu (a) \sucs \nu (c)$ whence 
$\nu (x) \crit{A} \nu (c)$.
In the third case we have 
$\nu (x) \crit{A} \nu (a) \sucs \nu (c)$, whence
$\nu (x) \crit{A} \nu (c)$.

\item
Again, $G'$ is closed off under the frame conditions with Lemma
\ref{lemm:ilmextension}.
Clearly, $\langle a,b, C\rhd D\rangle$ is not a deficiency on $G$.
\end{enumerate}
\medskip

{\bf Rounding up}
One of our invariants is just the \ilm frame condition. Clearly this
invariant is preserved under taking unions of bounded chains. The closure
satisfies the invariants.
\end{proof}

\subsection{Admissible rules}\label{subs:admissibleRules}

With the completeness at hand, a lot of reasoning about \ilm 
gets easier. This holds in 
particular for derived/admissible rules of \ilm.
In the following lemma, we will use the completeness theorem to obtain models. Most of the times these models will be glued above a fresh new world to obtain new models with the desired properties.

\begin{lemma}\label{lemm:derivablerules}
\ \ 
\begin{itemize}
\item[$(\romannumeral 1)$]
$\ilm \vdash \Box A \Leftrightarrow \ilm \vdash A$

\item[$(\romannumeral 2)$]
$\ilm \vdash \Box A \vee \Box B \Leftrightarrow \ilm \vdash \Box A \mbox{ or }
\ilm \vdash \Box B$

\item[$(\romannumeral 3)$]
$\ilm \vdash  A\rhd B \Leftrightarrow \ilm \vdash A \rightarrow B \vee \Diamond B$.

\item[$(\romannumeral 4)$]
$\ilm \vdash A \rhd B \Leftrightarrow \ilm \vdash \Diamond A \rightarrow
\Diamond B$

%

\item[$(\romannumeral 5)$]
Let $A_i$ be formulae such that 
$\ilm \not \vdash \neg A_i$. Then \\
$
\ilm \vdash \bigwedge \Diamond A_i \rightarrow A\rhd B  
\Leftrightarrow \ilm \vdash A\rhd B.
$
\item[$(\romannumeral 6)$]
$\ilm \vdash A \vee \Diamond A \Leftrightarrow \ilm \vdash \Box 
\bot \rightarrow A$

\item[$(\romannumeral 7)$]
$\ilm \vdash \top \rhd A \Leftrightarrow \ilm \vdash \Box 
\bot \rightarrow A$

\end{itemize}
\end{lemma}

\begin{proof}
$(\romannumeral 1)$. $\ilm \vdash A \Rightarrow \ilm \vdash \Box A$ 
by necessitation. Now suppose $\ilm \vdash \Box A$. We want to
see $\ilm \vdash A$. Thus, we take an arbitrary model
$M=\langle W,R,S,\Vdash \rangle$ and world $m\in M$. If there is an 
$m_0$ with $M\models m_0R m$, then $M,m_0 \Vdash \Box A$, whence
$M,m \Vdash A$. If there is no such $m_0$, we define (we may assume $m_0 \notin W$)
\[
\begin{array}{ll}
M'\eqbydef &\langle
W \cup \{ m_0 \}, 
R \cup \{  \langle m_0,w \rangle \mid w\in W\}, \\
\ & \ 
S \cup \{ \langle m_0 , x,y \rangle \mid \langle x,y \rangle \in R 
\mbox{ or } x{=}y\in W\},
\Vdash \rangle.
\end{array}
\]
Clearly, $M'$ is an \ilm-model too (the \ilm frame conditions
in the new cases follows from the transitivity of $R$), 
whence $M',m_0\Vdash \Box A$ and
thus $M',m \Vdash A$. By the construction of $M'$ and by Lemma \ref{lemm:gesulem}
we also get $M,m\Vdash A$.

$(\romannumeral 2)$.''$\Leftarrow$'' is easy. For the other
direction we assume $\ilm \not \vdash \Box A$ and $\ilm \not \vdash \Box B$
and set out to prove $\ilm \not \vdash \Box A \vee \Box B$.
By our assumption and by completeness, we find 
$M_0,m_0 \Vdash \Diamond \neg A$ and 
$M_1,m_1 \Vdash \Diamond \neg B$. We define (for some $r\notin W_0 \cup W_1$)
\[
\begin{array}{ll}
M\eqbydef & 
\langle W_0 \cup W_1 \cup \{ r\}, 
R_0\cup R_1\cup \{ \langle r,x \rangle \mid x \in W_0 \cup W_1  \},\\
\ & \ 
S_0\cup S_1\cup \{  \langle r,x,y \rangle \mid x{=}y {\in} 
W_0 \cup W_1 \mbox{ or }
\langle x,y \rangle {\in} R_0  \mbox{ or } \langle x,y \rangle {\in} R_1\},
\Vdash \rangle.
\end{array}
\]
Now, $M$ is an \ilm-model and 
$M,r \Vdash \Diamond \neg A \wedge \Diamond \neg B$ as is
easily seen by Lemma \ref{lemm:gesulem}. By soundness we get 
$\ilm \not \vdash \Box A \vee \Box B$.

$(\romannumeral 3)$.''$\Leftarrow$'' goes as follows.
$\vdash A \rightarrow B \vee \Diamond B \Rightarrow 
\vdash \Box (A \rightarrow B \vee \Diamond B) \Rightarrow
\vdash A \rhd B \vee \Diamond B \Rightarrow
\vdash A \rhd B$. For the other direction, suppose that
$\not \vdash A \rightarrow B \vee \Diamond B$. Thus, we
can find a model $M=\langle W,R,S,\Vdash \rangle$ 
and $m\in M$ with $M,m \Vdash A \wedge \neg B \wedge \Box \neg B$.
We now define (with $r\notin W$)
\[
\begin{array}{ll}
M'\eqbydef & 
\langle W \cup \{ r\}, 
R \cup \{ \langle r,x \rangle \mid x{=}m \mbox{ or } \langle m,x \rangle \in R \}
,\\
\ &
S \cup \{  \langle r,x,y \rangle \mid 
(x{=}y \mbox{ and }(
\langle m,x \rangle {\in} R \mbox{ or } x{=}m)) \mbox{ or } \langle m,x \rangle, \langle x,y 
\rangle {\in} R\}
, \Vdash \rangle.
\end{array}
\]
It is easy to see that $M'$ is an \ilm-model. By Lemma \ref{lemm:gesulem}
we see that $M',x\Vdash \varphi$ iff $M, x \Vdash \varphi$ for 
$x\in W$. It is also not hard to see that $M',r \Vdash \neg (A\rhd B)$.
For, we have $rRm\Vdash A$. By definition, 
$mS_ry \rightarrow (m{=}y \vee mRy)$ whence $y\not \Vdash B$.

$(\romannumeral 4)$.
By the $\sf J4$ axiom, we get one direction for free. For the 
other direction we reason as follows. Suppose 
$\ilm \nvdash A\rhd B$. Then we can find a model 
$M = \langle W,R,S, \Vdash \rangle$ and a world $l$
such that $M,l \Vdash \neg (A \rhd B)$. 
As $M,l \vdash \neg (A\rhd B)$, w can find some
$m \in M$ with 
$lRm\Vdash A \wedge \neg B \wedge \Box \neg B$.
We now define (with $r\notin W$)
\[
\begin{array}{ll}
M'\eqbydef & 
\langle W \cup \{ r\}, 
R \cup \{ \langle r,x \rangle \mid x{=}m \mbox{ or } \langle m,x \rangle \in R \}
,\\
\ &
S \cup \{  \langle r,x,y \rangle \mid 
(x{=}y \mbox{ and }(
\langle m,x \rangle {\in} R \mbox{ or } x{=}m)) \mbox{ or } \langle m,x \rangle, \langle x,y 
\rangle {\in} R\}
, \Vdash \rangle.
\end{array}
\]
It is easy to see that $M'$ is an \ilm-model.
Lemma \ref{lemm:gesulem}
and general knowledge about \ilm tells us that the 
generated submodel from $l$ is a witness to the fact that
$\ilm \nvdash \Diamond A \rightarrow \Diamond B$.\footnote{This
proof is similar to the proof of $(\romannumeral 3)$. 
However, it is not the case that one of the two follows easily from the other.}


$(\romannumeral 5)$.
The ''$\Leftarrow$'' direction is easy. For the other direction we
reason as follows.\footnote{By a similar reasoning we can prove 
$\vdash \bigwedge \neg (C_i \rhd D_i) \rightarrow A\rhd B 
\Leftrightarrow \vdash A \rhd B$.}

We assume that $\not \vdash A\rhd B$ and set out to prove
$\not \vdash \bigwedge \Diamond A_i \rightarrow A\rhd B$.
As $\not \vdash A\rhd B$, we can find $M,r\Vdash \neg (A \rhd B)$.
By Lemma \ref{lemm:gesulem} we may assume that $r$ is a root of
$M$. For all $i$, we assumed $\not \vdash \neg A_i$, whence we can
find rooted models $M_i,r_i \Vdash A_i$. As in the other cases, we define a model
$\tilde M$ that arises by gluing $r$ under all the $r_i$. Clearly
we now see that 
$\tilde M , r \Vdash \bigwedge \Diamond A_i \wedge \neg (A \rhd B)$.

$(\romannumeral 6)$. First, suppose that 
$\ilm \vdash \Box \bot \rightarrow A$. Then, from 
$\ilm \vdash \Box \bot \vee \Diamond \top$, the observation
that
$\ilm \vdash \Diamond \top \leftrightarrow \Diamond \Box \bot$
and our assumption, we
get $\ilm \vdash A \vee \Diamond A$.

For the other direction, we suppose that 
$\ilm \not \vdash \Box \bot \rightarrow A$. Thus, we have a counter model
$M$ and some $m\in M$ with $m\Vdash \Box \bot, \neg A$. Clearly,
at the submodel generated from $m$, that is, a single point, we see
that $\neg A \wedge \Box \neg A$ holds. Consequently 
$\ilm \neg \vdash A \vee \Diamond A$.

$(\romannumeral 7)$.
This follows immediately from $(\romannumeral 6)$ 
and $(\romannumeral 3)$.

\end{proof}

Note that, as \ilm is conservative over \gl, all of the above statements
not involving $\rhd$ also hold for \gl. The same holds for 
derived statements. For example, from Lemma \ref{lemm:derivablerules}
we can combine $(\romannumeral 3)$ and 
$(\romannumeral 4)$ to obtain 
$\ilm \vdash A \rightarrow B \vee \Diamond B \Leftrightarrow 
\ilm \vdash \Diamond A \rightarrow \Diamond B$. Consequently, the same
holds true for \gl.

\subsection{Decidability}

%
It is well known that \ilm has the finite model property. 
It is not hard to re-use worlds in the presented construction
method so that we would end up with a finite counter model.
Actually, this is precisely what has been done in \cite{joo98}.
In that paper, one of the invariants was
``there are no deficiencies''. We have chosen not to include
this invariant in our presentation, as this omission 
simplifies the presentation. Moreover, for our purposes
the completeness without the finite model property obtained via
our construction method suffices.

Our purpose to include a new proof of the well known completeness 
of \ilm is twofold. On the one hand the new proof serves well to expose the
construction method. On the other hand, it is an indispensable
ingredient in proving Theorem \ref{theo:ilmessig}.

\section{Essentially $\Sigma_1$-sentences of \ilm}\label{sect:essigma}
In this section we
will answer the question which modal interpretability sentences are in theories $T$  
provably $\Sigma_1$ for any realization. We call these
sentences essentially $\Sigma_1$-sentences. We shall answer the 
question only for $T$ an essentially reflexive theory.

This question has been solved for provability logics by Visser 
in \cite{Vis93b}. In \cite{Jopia96}, de Jongh and Pianigiani
gave an alternative solution by using the logic \ilm.
Our proof shall use their proof method.

We will perform our argument
fully in \ilm. It is very tempting to think that our result would be an
immediate corollary from for example \cite{Goris03}, \cite{Japa94} or
\cite{igna93}. This would be the case, if a construction method were 
worked out for the logics from these respective papers. In \cite{Goris03} 
a sort of a  construction method is indeed worked out. This construction method 
should however be a bit sharpened to suit our purposes. Moreover that sharpening
would essentially reduce to the solution we present here.

\subsection{Model construction}

Throughout this subsection, unless mentioned otherwise, 
$T$ will be an essentially reflexive recursively 
enumerable arithmetical theory. By Theorem \ref{theo:shav} we thus know that
$\intl{T}=\ilm$. Let us first say more precisely what we mean by an 
essentially $\Sigma_1$-sentence.

\begin{definition}
A modal sentence $\varphi$ is called an essentially $\Sigma_1$-sentence with respect to a theory $T$,
if $\forall *\ \varphi^*\in \Sigma_1(T)$. Likewise, a formula
$\varphi$ is essentially $\Delta_1$ if 
$\forall *\ \varphi^*\in \Delta_1(T)$
\end{definition}
If $\varphi$ is an essentially $\Sigma_1$-formula for $T$ we will 
also write $\varphi \in \Sigma_1(T)$. Analogously for $\Delta_1(T)$. For the rest of this section, $T$ will always be a theory that has \ilm as its interpretability logic thereby making explicit reference to $T$ unnecessary as we shall see.

\begin{theorem}\label{theo:delta}
Modulo modal logical equivalence, there exist just two
essentially $\Delta_1$-formulas in the language of \ilm. That is, $\Delta_1(T)= \{ \top, \bot\}$.
\end{theorem}

\begin{proof}
Let $\varphi$ be a modal formula. If $\varphi \in \Delta_1(T)$, then, by 
provably $\Sigma_1$-completeness, both 
$\forall *\ T\vdash \delta^* \rightarrow \Box \delta^*$ and
$\forall *\ T\vdash \neg \delta^* \rightarrow \Box \neg \delta^*$.
Consequently
$\forall *\ T\vdash \Box \delta^* \vee \Box \neg \delta^*$.
Thus, $\forall *\ T\vdash (\Box \delta \vee \Box \neg \delta)^*$
whence
$\ilm \vdash \Box \delta \vee \Box \neg \delta$. By Lemma \ref{lemm:derivablerules}
we see that $\ilm \vdash \delta$ or $\ilm \vdash \neg \delta$.
\end{proof}

We proved Theorem \ref{theo:delta} for the interpretability logic of 
essentially reflexive theories. It is not hard to see that the theorem also
holds for finitely axiomatizable theories. The only ingredients that we need
to prove this are [$\ilp \vdash \Box A \vee \Box B$ iff. 
$\ilp \vdash \Box A$ or $\ilp \vdash \Box B$] and
[$\ilp \vdash \Box A$ iff. $\ilp \vdash A$]. As these two admissible rules also
hold for \gl, we see that Theorem \ref{theo:delta} also holds for \gl.

The following lemma is the only arithmetical ingredient in our classification of the essentially $\Sigma_1$ formulas in the language of \ilm. 

\begin{lemma}\label{lemm:sigmasufficient}
If $\varphi \in \Sigma_1 (T)$, then, for any $p$ and $q$, we have
$\ilm \vdash p\rhd q \rightarrow p\wedge \varphi \rhd q\wedge \varphi$.
\end{lemma}
Before we come to prove the main theorem of this section, we first
need an additional lemma.

\begin{lemma}\label{lemm:under}
Let $\Delta_0$ and $\Delta_1$ be maximal \ilm-consistent sets.
There is a maximal \ilm-consistent set
$\Gamma$ such that $\Gamma \sucs \Delta_0, \Delta_1$.
\end{lemma}

\begin{proof}
We show that 
$\Gamma' \eqbydef \{ \Diamond A \mid A \in \Delta_0\} \cup 
\{ \Diamond B \mid B \in \Delta_1 \}$ is consistent. Assume for a 
contradiction that $\Gamma'$ were not consistent. Then, by compactness,
for finitely many $A_i$ and $B_j$,
\[
\bigwedge_{A_i \in \Delta_0}\Diamond A_i \wedge
\bigwedge_{B_j \in \Delta_1}\Diamond B_j \vdash \bot
\]
or equivalently
\[
\vdash \bigvee_{A_i \in \Delta_0}\Box \neg A_i \vee 
\bigvee_{B_j \in \Delta_1}\Box \neg B_j.
\]
By Lemma \ref{lemm:derivablerules} we see that then either 
$\vdash \neg A_i$ for some $i$, or 
$\vdash \neg B_j$ for some $j$.
This contradicts the consistency of $\Delta_0$ and
$\Delta_1$.
\end{proof}

With this lemma and by postponing the hard work to Subsection{subs:sigmaLemma} we can now prove the main theorem of this section.

\begin{theorem}\label{theo:ilmessig}
$\varphi \in \Sigma_1 (T) \Leftrightarrow \ilm \vdash 
\varphi \leftrightarrow \bigvee_{i\in I}\Box C_i$ for some
$\{  C_i\}_{i\in I}$.
\end{theorem}

\begin{proof}
Let $\varphi$ be a formula that is not equivalent to a disjunction of
$\Box$-formulas. According to Lemma \ref{lemm:essilm} we can find 
\mcs's $\Delta_0$ and $\Delta_1$ with 
$\varphi \in \Delta_0 \boxin \Delta_1 \ni \neg \varphi$. By
Lemma \ref{lemm:under} we find a $\Gamma \prec \Delta_0, \Delta_1$.
We define:
\[
G\eqbydef
\langle 
\{ m_0,l,r\}, 
\{ \langle m_0 , l\rangle ,\langle m_0,r \rangle \},
\{ \langle m_0 ,l,r\rangle  \},
\{ \langle m_0 , \Gamma \rangle , \langle l, \Delta_0  \rangle , 
\langle r, \Delta_1 \rangle\}
\rangle.
\] 
We will apply a slightly generalized version of the 
main lemma to this quasi-\ilm-frame $G$. 
The finite set \adset{D} of sentences is the 
smallest set of sentences that contains $\varphi$ and that is closed under
taking subformulas and single negations. The invariants are the following.
\[
\mathcal{I} \eqbydef \left \{ \begin{array}{l} xRy \wedge x \neq m_0
\rightarrow \exists \,  A {\in}((\nu (y)\setminus \nu (x))\cap 
\{  \Box D \mid D \in \adset{D}\} ) \\ 
uS_xvRw\rightarrow uRw
\end{array}\right.
\]
In the proof of Theorem \ref{theo:ilmcomplete} we have seen that we
can eliminate both problems and deficiencies while conserving the invariants. 
The main lemma now gives us an \ilm-model $M$ with 
$M,l\Vdash \varphi$, $M,r\Vdash \neg \varphi$ and $lS_{m_0}r$. We now
pick two fresh variables $p$ and $q$. We define $p$ to be true only at 
$l$ and $q$ only at $r$. Clearly 
$m_0 \Vdash \neg (p\rhd q \rightarrow p\wedge \varphi \rhd q \wedge \varphi)$,
whence by Lemma \ref{lemm:sigmasufficient} we get $\varphi \notin \Sigma_1 (T)$.

\end{proof}

For finitely axiomatized theories $T$, our theorem does not hold, as also
$A\rhd B$ is $T$-essentially $\Sigma_1$. The following theorem says that
in this case,
$A\rhd B$ is under any $T$-realization actually equivalent to a
special $\Sigma_1$-sentence.

\begin{theorem}\label{theo:boxilp}
Let $T$ be a finitely axiomatized theory. For all arithmetical formulae 
$\alpha$, $\beta$ there exists a formula $\rho$ with
\[
T\vdash \alpha \rhd_T \beta \leftrightarrow \Box_T \rho .
\]
\end{theorem}

\begin{proof}
The proof is a direct corollary of the so-called FGH-theorem. 
(See \cite{viss:faith02} for an exposition of the FGH-theorem.)
We take $\rho$ satisfying the following fixed point equation.
\[
T\vdash \rho \leftrightarrow ((\alpha \rhd_T \beta) \leq \Box_T \rho)
\]
By the proof of the FGH-theorem, we now see that 
\[
T\vdash ((\alpha \rhd_T \beta ) \vee \Box_T \bot) \leftrightarrow \Box_T \rho .
\]
But clearly 
$T \vdash ((\alpha \rhd_T \beta ) \vee \Box_T \bot) \leftrightarrow \alpha \rhd_T \beta$.
\end{proof}

\subsection{The $\Sigma$-lemma}\label{subs:sigmaLemma}
We can say that the proof of Theorem \ref{theo:ilmessig} contained 
three main ingredients; Firstly, the main lemma; Secondly the modal 
completeness theorem for \ilm via the construction method and; Thirdly the
$\Sigma$-lemma. In this subsection we will prove the $\Sigma$-lemma and
remark that it is in a sense optimal.

\newcommand{\boxen}{\Box_{\vee}}
\newcommand{\boxcon}{\Box_{\textup{con}}}
\newcommand{\boxconfin}{\boxcon^{\textup{fin}}}
\newcommand{\yfin}{Y^{\textup{fin}}}
\newcommand{\sfin}{S^{\textup{fin}}}
\newcommand{\E}{\exists}
\newcommand{\A}{\forall}

\begin{lemma}\label{lemm:essilm}
If $\varphi$ is a formula not equivalent to a disjunction of $\Box$-formulas.
Then there exist 
maximal \ilx-consistent sets $\Delta_0$, $\Delta_1$ such that
$\varphi \in \Delta_0 \boxin \Delta_1 \ni \neg \varphi$.
\end{lemma}


\begin{proof}
As we shall see, the reasoning below holds not only for \extil{X}, but for
any extension of \gl.
We define
\begin{eqnarray*}
	\boxen	&\eqbydef & 
\{ \bigvee_{0\leq i<n}\Box D_i \mid n\geq 0 ,\textrm{each }D_i\textrm{ an \extil{X}-formula}\},\\
	\boxcon 	&\eqbydef & \{ Y \subseteq \boxen \mid
		\{ \neg \varphi \} + Y \textrm{ is consistent and maximally such}\}.
\end{eqnarray*}
\noindent
Let us first observe a useful property of the sets $Y$ in $\boxcon$.

\begin{equation}\label{useful_property}
\bigvee_{i=0}^{n-1}\sigma_i\in Y\Rightarrow \E \,  i{<}n\ \sigma_i\in Y.
\end{equation}

To see this, let $Y\in\boxcon$ and $\bigvee_{i=0}^{n-1}\sigma_i\in Y$.
Then for each $i{<}n$ we have $\sigma_i\in\boxen$ and for some
$i{<}n$ we must have  $\sigma_i$ consistent with $Y$ (otherwise 
$\{ \neg \varphi \} + Y$ would prove
$\bigwedge_{i=0}^{n-1}\neg\sigma_i$ and be inconsistent).
And thus by the maximality of $Y$ we must have that some $\sigma_i$ is in $Y$.
This establishes \eqref{useful_property}.

\begin{claim}
For some $Y \in \boxcon$ the set
\[
\{\varphi\} + \{\neg\sigma \mid \sigma\in\boxen-Y \}
\]
is consistent.
\end{claim}

\begin{proof}[Proof of the claim]
Suppose the claim were false.
We will derive a contradiction with the assumption that 
$\varphi$ is not equivalent to a disjunction of $\Box$-formulas.
If the claim is false, then we can choose for each $Y \in \boxcon$ a finite
set $\yfin \subseteq \boxen - Y$ such that

\begin{equation}\label{f:inc0}
	\{ \varphi \} + \{ \neg \sigma \mid \sigma \in \yfin \}
\end{equation}
is inconsistent. 
Thus, certainly
for each $Y \in \boxcon$
\begin{equation}\label{equa:eenkant} \vdash \varphi \rightarrow
	\bigvee_{\sigma \in \yfin}\sigma.
\end{equation}
Now we will show that:

\begin{equation}\label{set:inconsistent1}
\{ \neg \varphi \} + \{ \bigvee_{\sigma \in \yfin} \sigma \mid Y \in \boxcon \} \textrm{ is inconsistent. }
\end{equation}
For, suppose (\ref{set:inconsistent1}) were not the case. 
Then for some $S \in \boxcon$
\[ \{ \bigvee_{\sigma \in \yfin} \sigma \mid Y \in \boxcon \} \subseteq S.\]
In particular we have
$\bigvee_{\sigma \in \sfin} \sigma \in S$.
But for all $\sigma \in \sfin$ we have $\sigma \not \in S$.
Now by \eqref{useful_property} we obtain a contradiction
and thus we have shown (\ref{set:inconsistent1}).

So we can select some finite $\boxconfin \subseteq \boxcon$ such that 

\begin{equation}\label{f:s0}
	\vdash(\bigwedge_{Y \in\boxconfin} \bigvee_{\sigma \in \yfin} \sigma)
\rightarrow \varphi.
\end{equation}
By (\ref{equa:eenkant}) we also have

\begin{equation}\label{f:s1}
	\vdash \varphi \rightarrow
		\bigwedge_{Y\in\boxconfin}\bigvee_{\sigma\in \yfin}\sigma.
\end{equation}
Combining (\ref{f:s0}) with (\ref{f:s1}) we get

\[
\vdash
		\varphi \leftrightarrow 	
		\bigwedge_{Y \in \boxconfin}\bigvee_{\sigma \in \yfin} \sigma .
\]
Bringing the right hand side of this equivalence in disjunctive 
normal form and distributing the $\Box$ over $\wedge$
we arrive at a contradiction with the assumption on $\varphi$.
\end{proof}
So, we have for some $Y \in \boxcon$ that both the sets

\begin{equation}\label{f:2}
\{ \varphi\} + \{ \neg \sigma \mid \sigma \in \Box_\vee -Y \}
\end{equation}

\begin{equation}\label{f:3}
\{ \neg \varphi \} + Y
\end{equation}
are consistent. The lemma follows by taking $\Delta_0$ and $\Delta_1$ 
extending (\ref{f:2}) and (\ref{f:3}) respectively.
\end{proof}

We have thus obtained 
$\varphi \in \Delta_0 \boxin \Delta_1 \ni \neg \varphi$
for some maximal \extil{X}-consistent sets $\Delta_0$ and 
$\Delta_1$. The relation $\boxin$ between $\Delta_0$ and $\Delta_1$
is actually the best we can get among the relations on \mcs's that we 
consider in this paper. We shall see that $\Delta_0 \sucs \Delta_1$
is not possible to get in general.

It is obvious that
that $p\wedge \Box p$ is not equivalent to
a disjunction of $\Box$-formulas. Clearly 
$p\wedge \Box p\in \Delta_0 \sucs \Delta_1 \ni \neg p \vee \Diamond \neg p$ 
is impossible. In a sense, this reflects the fact that there exist non trivial 
self-provers, as was shown by Kent (\cite{Kent73}), Guaspari (\cite{gua83}) and
Beklemishev (\cite{Bek93}). Thus, provable
$\Sigma_1$-completeness, that is 
$T\vdash \sigma \rightarrow \Box \sigma$ for $\sigma \in \Sigma_1(T)$, can 
not substitute Lemma \ref{lemm:sigmasufficient}.

\section{Self provers and $\Sigma_1$-sentences}

A self prover is a sentence $\varphi$ that implies its own provability.
That is, a sentence for which $\vdash \varphi \rightarrow \Box \varphi$,
or equivalently, 
$\vdash \varphi \leftrightarrow \varphi \wedge \Box \varphi$.
Self provers have been studied intensively amongst others by 
Kent (\cite{Kent73}),
Guaspari (\cite{gua83}), de Jongh and Pianigiani (\cite{Jopia96}).
It is easy to see that any $\Sigma_1(T)$-sentence is indeed
a self prover. We shall call such a self prover a 
\emph{trivial self prover}.

In \cite{gua83}, Guaspari has shown that there are many non-trivial 
self provers around. The most prominent example is probably 
$p\wedge \Box p$. But actually, any formula $\varphi$ will generate
a self prover $\varphi \wedge \Box \varphi$, as clearly
$\varphi \wedge \Box \varphi \rightarrow \Box (\varphi \wedge \Box \varphi)$.

\begin{definition}
A formula $\varphi$ is called a trivial self prover generator, we shall
write t.s.g., if $\varphi \wedge \Box \varphi$ is a trivial self prover.
That is, if $\varphi \wedge \Box \varphi \in \Sigma_1(T)$.
\end{definition}

Obviously, a trivial self prover is also a t.s.g. But there 
also exist other t.s.g.'s. The most prominent example is probably
$\Box \Box p \rightarrow \Box p$. A natural question is to ask for 
an easy characterization of t.s.g.'s.  In this  section we will 
give such a characterization for \gl. All results presented here are new results. In the rest of this section,
$\vdash$ will stand for derivability in \gl. We shall often write $\Sigma$
instead of $\Sigma_1$.

We say that a formula $\psi$ is $\Sigma$ in \gl, and write
$\Sigma (\psi)$, if for any theory $T$ which has \gl as its 
provability logic, we have that $\forall *\ \psi^* \in \Sigma_1(T)$.

\begin{theorem}\label{theo:disjunctive}
We have that $\Sigma(\varphi \wedge \Box \varphi)$ in \gl if and only if the 
following condition is satisfied.

For all formulae $A_l$, $\varphi_l$ and $C_m$ satisfying 
\ref{item:equi}, \ref{item:ired} and \ref{item:nonempty} we have that
$\vdash \varphi \wedge \Box \varphi \leftrightarrow \dis_m \Box C_m$. Here 
\ref{item:equi}-\ref{item:nonempty} are the following conditions.

\begin{enumerate}
\item \label{item:equi}
$\vdash \varphi \leftrightarrow \dis_l (\varphi_l \wedge \Box A_l)
\vee \dis_m \Box C_m$

\item \label{item:ired}
$\not \vdash \Box A_l \rightarrow \varphi$ for all $l$

\item \label{item:nonempty}
$\varphi_l$ is a non-empty conjunction of literals and $\Diamond$-formulas.
\end{enumerate}

\end{theorem}

\begin{proof}
The $\Leftarrow$ direction is the easiest part. 
We can always 
find an equivalent of $\varphi$ that satisfies
\ref{item:equi}, \ref{item:ired} and \ref{item:nonempty}.
Thus, by 
assumption, $\varphi \wedge \Box \varphi$ can be written as the disjunction
of $\Box$-formulas and hence $\Sigma (\varphi \wedge \Box \varphi)$.

For the $\Rightarrow$ direction we reason as follows. Suppose we
can find $\varphi_l$, $A_l$ and $C_m$ such that 
\ref{item:equi}, \ref{item:ired} and \ref{item:nonempty} hold, but
\[
\not \vdash \varphi \wedge \Box \varphi \leftrightarrow \dis_m \Box C_m. \ \ (*)
\]
We can take now $T=\pa$ and reason as follows.
As clearly 
$\vdash \dis_m \Box C_m \rightarrow \varphi \wedge \Box \varphi$,
our assumption $(*)$ reduces to 
$\not \vdash \varphi \wedge \Box \varphi \rightarrow \dis_m \Box C_m$.
Consequently $\dis_l (\varphi_l \wedge \Box A_l)$ can not be
empty, and for some $l$ and some rooted \gl-model $M,r$  with 
root $r$, we 
have $M,l \Vdash \Box A_l \wedge \varphi_l$.

We shall now see that 
$\not \vdash \neg \varphi \wedge \Box \varphi \rightarrow \Diamond \neg A_l$.
For, suppose for a contradiction that 
\[
\vdash  \neg \varphi \wedge \Box \varphi \rightarrow \Diamond \neg A_l .
\]
Then also 
$\vdash \Box A_l \rightarrow (\Box \varphi \rightarrow \varphi)$, whence
$\vdash \Box A_l \rightarrow \Box(\Box \varphi \rightarrow \varphi) \rightarrow
\Box \varphi$. And by 
$\Box A_l \rightarrow (\Box \varphi \rightarrow \varphi)$ again, we get 
$\vdash \Box A_l \rightarrow \varphi$ which contradicts \ref{item:ired}.
We must conclude that indeed
$\not \vdash \neg \varphi \wedge \Box \varphi \rightarrow \Diamond \neg A_l$,
and thus we have a rooted tree model $N,r$ for \gl 
with $N,r \Vdash \neg \varphi, \Box \varphi , \Box A_l$.

We can 
now ``glue'' a world $w$ below $l$ and $r$, set $lS_wr$ and 
consider the smallest \ilm-model extending this. We have depicted this
construction in Figure \ref{pict:selfprover}.
\begin{figure}
\begin{center}
\input{selfprover.pstex_t}
\end{center}
\caption{T.s.g.'s}\label{pict:selfprover}
\end{figure}
Let us also give a precise definition. If 
$M\eqbydef \langle W_0,R_0,\Vdash_0\rangle$ and 
$N\eqbydef \langle W_1,R_1,\Vdash_1\rangle$, then we define 
\[
\begin{array}{ll}
L \eqbydef & \langle W_0 \cup W_1, 
R_0 \cup R_1 \cup \{ \langle w,x \rangle \mid x \in W_0 \cup W_1\}
\cup \{ \langle l,y \rangle \mid N \models rRy\},\\
 \ &  \{ \langle w,l,r \rangle  \} \cup 
\{ \langle x,y,z\rangle \mid L \models xRyR^*z\}, \Vdash_0 \cup \Vdash_1 \rangle .
\end{array}
\]
We observe that, by Lemma \ref{lemm:gesulem} 
$L,r \Vdash \Box \varphi \wedge \Box A_l \wedge \neg \varphi$
and $L\models rRx \Rightarrow L,x\Vdash \varphi \wedge A_l$.
Also, if $L\models lRx$, then $L,x \Vdash \varphi \wedge A_i$, whence
$L,l\Vdash \Box \varphi \wedge \Box A_l$. As 
$M,l\Vdash \varphi_l$ and $\varphi_l$ only contains literals and
and diamond-formulas, we see that $L,l\Vdash \varphi_l$, 
whence $L,l \Vdash \varphi \wedge \Box \varphi$. As 
$L,r\Vdash \neg \varphi \wedge \Box \varphi$ we see that
$L,w \Vdash \neg \Box (\varphi \wedge \Box \varphi )$.

As in the proof of Theorem \ref{theo:ilmessig}, we can
take some fresh $p$ and $q$ and define $p$ to 
hold only at $l$ and $q$ to hold only at $r$. Now, 
clearly 
$w \not \Vdash p\rhd q \rightarrow p \wedge (\varphi \wedge \Box \varphi)\rhd
q \wedge (\varphi \wedge \Box \varphi)$, whence, by Lemma
\ref{lemm:sigmasufficient} we conclude 
$\neg \Sigma (\varphi \wedge \Box \varphi)$.
\end{proof}

The above reasoning showed that $\Sigma(\varphi \wedge \Box \varphi)$ is not a sufficient condition 
for $\Sigma(\varphi)$ to hold. We shall see that even $\Sigma(\varphi \wedge \Box \varphi) \ \ \wedge \ \ 
\Sigma(\varphi \wedge \Box \neg \varphi)$ is not a sufficient condition 
for $\Sigma(\varphi)$ to hold.

Thus, to conclude this section, we remain in \gl and shall settle the 
question for which
$\varphi$ we have that 
\[
\Sigma (\varphi \wedge \Box \varphi ) \ \& \ \Sigma (\varphi \wedge \Box \neg \varphi) 
\Rightarrow \Sigma (\varphi ). \ \ (\dag)
\]
We shall see that this question is non-trivial and that it can be reduced to the characterization of 
t.s.g.'s. Again the easiest non-trivial example satisfying $(\dag)$ will be $\Box \Box p \to \Box p$.

\begin{lemma}\label{lemm:almostloeb}
\[
\begin{array}{c}
\mbox{For some (possibly empty) } \dis_i \Box C_i \mbox{ we have }
\vdash \varphi \wedge \Box \neg \varphi 
\leftrightarrow \dis_i\Box C_i \\
\mbox{iff.}\\
\vdash \Box \bot \rightarrow \varphi 
\mbox{ \ \  or \ \  } \vdash \neg \varphi
\end{array}
\]
\end{lemma}

\begin{proof}
For non-empty $\dis_i \Box C_i$ we have the following.
\[
\begin{array}{ll}
\vdash \varphi \wedge \Box \neg \varphi \leftrightarrow   \dis_i\Box C_i  & 
\Rightarrow \\
\vdash \Diamond (\varphi \wedge \Box \neg \varphi ) \leftrightarrow \Diamond 
(\dis_i\Box C_i)    & 
\Rightarrow \\
\vdash \Diamond \varphi \leftrightarrow \Diamond \top  & 
\Rightarrow \\
\vdash \Box \bot \rightarrow \varphi     
\end{array}
\]
Here, the final step in the proof comes from Lemma \ref{lemm:derivablerules}.

On the other hand, if $\vdash \Box \bot \rightarrow \varphi$, we
see that $\vdash \neg \varphi \rightarrow \Diamond \top$ and 
thus $\Box \neg \varphi \rightarrow \Box \bot$, whence
$\vdash \varphi \wedge \Box \neg \varphi \leftrightarrow \Box \bot$.

In case of the empty disjunction we get 
$\vdash \varphi \wedge \Box \neg \varphi \leftrightarrow \bot$.
Then also $\vdash \Box \neg \varphi \rightarrow \neg \varphi$ and
by L\"ob $\vdash \neg \varphi$. And conversely, if 
$\vdash \neg \varphi$, then 
$\vdash \varphi \wedge \Box \neg \varphi \leftrightarrow \bot$, and
$\bot$ is just the empty disjunction.

The proof actually gives some additional information. If 
$\Sigma (\varphi \wedge \Box \neg \varphi)$ then either 
($\vdash \neg \varphi$ and 
$\vdash (\varphi \wedge \Box \neg \varphi) \leftrightarrow \bot$),
or ($\vdash \Box \bot \rightarrow \varphi$ and 
$\vdash (\varphi \wedge \Box \neg \varphi) \leftrightarrow \Box \bot$).
\end{proof}

\begin{lemma}
\[
\begin{array}{c}
\Sigma (\varphi \wedge \Box \varphi ) \wedge \Sigma (\varphi \wedge \Box \neg \varphi )
\Rightarrow \Sigma (\varphi )\\
\mbox{ iff. }\\
\Sigma (\varphi \wedge \Box \varphi ) \Rightarrow \Sigma (\varphi ) 
\mbox{ or } \vdash \varphi \rightarrow \Diamond \top
\end{array}
\]
\end{lemma}

\begin{proof}
$\Uparrow$. Clearly, if $\Sigma (\varphi \wedge \Box \varphi ) \Rightarrow \Sigma (\varphi )$,
also 
$\Sigma (\varphi \wedge \Box \varphi ) \wedge \Sigma (\varphi \wedge \Box \neg \varphi )
\Rightarrow \Sigma (\varphi )$. Thus, suppose $\vdash \varphi \rightarrow \Diamond \top$,
or put differently $\vdash \Box \bot \rightarrow \neg \varphi$.
If now $\vdash \neg \varphi$, then clearly $\Sigma (\varphi)$, whence
$\Sigma (\varphi \wedge \Box \varphi ) \wedge \Sigma (\varphi \wedge \Box \neg \varphi )
\Rightarrow \Sigma (\varphi )$, so, we may assume that 
$\nvdash \neg \varphi$.
It is clear that now $\neg \Sigma (\varphi \wedge \Box \neg \varphi )$. For, suppose
$\Sigma (\varphi \wedge \Box \neg \varphi )$, then by Lemma \ref{lemm:almostloeb}
we see $\vdash \Box \bot \rightarrow \varphi$, whence $\vdash \Diamond \top$. Quod non.
Thus, $\vdash \Box \bot \rightarrow \neg \varphi \Rightarrow \neg \Sigma (\varphi 
\wedge \Box \neg \varphi )$ and thus certainly
$\Sigma (\varphi \wedge \Box \varphi ) \wedge \Sigma (\varphi \wedge \Box \neg \varphi )
\Rightarrow \Sigma (\varphi )$.

$\Downarrow$. Suppose 
$\Sigma (\varphi \wedge \Box \varphi ) \wedge \neg \Sigma (\varphi)$
and $\nvdash \Box \bot \rightarrow \neg \varphi$. To obtain our result, we only have
to prove $\Sigma (\varphi \wedge \Box \neg \varphi )$.

As $\nvdash \Box \bot \rightarrow \neg \varphi$, also 
$\nvdash \neg \varphi \vee \Diamond \neg \varphi$. Thus, under the assumption
that $\Sigma (\varphi \wedge \Box \varphi)$, we can find (a non-empty collection of)
$C_i$ with $\vdash \varphi \wedge \Box \varphi \leftrightarrow \dis_i \Box C_i$.
In this case, clearly 
$\vdash \Box \bot \rightarrow \dis_i \Box C_i \rightarrow \varphi$, whence, by 
Lemma \ref{lemm:almostloeb} we conclude $\Sigma (\varphi \wedge \Box \neg \varphi)$.
\end{proof}




\bibliographystyle{plain}


\begin{thebibliography}{100}

\bibitem{arec:inte98}
C.~Areces, D.~de~Jongh, and E.~Hoogland.
\newblock The interpolation theorem for{ ${\sf {I}{L}}$ and ${\sf {I}{L}{P}}$}.
\newblock In {\em Proceedings of AiML98. Advances in Modal Logic}, Uppsala.
  Sweden, October 1998. Uppsala University.

\bibitem{Bek93}
L.D. Beklemishev.
\newblock On the complexity of arithmetic interpretations of modal formulae.
\newblock {\em Archive for Mathematical Logic}, 32:229--238, 1993.

\bibitem{bera:inte90}
A.~Berarducci.
\newblock The interpretability logic of {P}eano arithmetic.
\newblock {\em Journal of Symbolic Logic}, 55:1059--1089, 1990.

\bibitem{Bla01}
P.~Blackburn, M.~de Rijke, and Y.~Venema.
\newblock {\em {M}odal {L}ogic}.
\newblock {C}amebridge {U}niversity {P}ress, 2001.

\bibitem{Boo93}
G.~Boolos.
\newblock {\em The {L}ogic of {P}rovability}.
\newblock Cambridge University Press, Cambridge, 1993.

\bibitem{DJ}
D.~de~Jongh and G.~Japaridze.
\newblock The {L}ogic of {P}rovability.
\newblock In S.R. Buss, editor, {\em Handbook of Proof Theory.
  \textnormal{Studies in Logic and the Foundations of Mathematics, Vol.137.}},
  pages 475--546. Elsevier, Amsterdam, 1998.

\bibitem{Jopia96}
D.~de~Jongh and D.~Pianigiani.
\newblock Solution of a problem of {D}avid {G}uaspari.
\newblock {\em Studia Logica}, 1996.

\bibitem{JoVe90}
D.H.J. de~Jongh and F.~Veltman.
\newblock Provability logics for relative interpretability.
\newblock In {\em \cite{heyt:malo90}}, pages 31--42, 1990.

\bibitem{jonvelt99}
D.H.J. de~Jongh and F.~Veltman.
\newblock Modal completeness of {IL${\sf W}$}.
\newblock In J.~Gerbrandy, M.~Marx, M.~Rijke, and Y.~Venema, editors, {\em
  Essays dedicated to Johan van Benthem on the occasion of his 50th birthday}.
  Amsterdam University Press, Amsterdam, 1999.

\bibitem{dejo:expl91}
D.H.J. de~Jongh and A.~Visser.
\newblock Explicit fixed points in interpretability logic.
\newblock {\em Studia Logica}, 50:39--50, 1991.

\bibitem{Goris03}
E.~Goris.
\newblock Extending {ILM} with an operator for {$\Sigma_1$}--ness.
\newblock Illc prepublication series, University of Amsterdam, 2003.

\bibitem{gua83}
D.~Guaspari.
\newblock Sentences implying their own provability.
\newblock {\em Journal of Symbolic Logic}, 48:777--789, 1983.

\bibitem{HP}
P.~H\'ajek and P.~Pudl\'ak.
\newblock {\em Metamathematics of {F}irst {O}rder {A}rithmetic}.
\newblock Springer-{V}erlag, Berlin, Heidelberg, New York, 1993.

\bibitem{hiho02}
{R}. Hirsch and I.~Hodkinson.
\newblock {\em Relation Algebras by Games}, volume 147 of {\em Studies in
  Logic}.
\newblock Elsevier, North-Holland, 2002.

\bibitem{HoMiVe01}
I.~Hodkinson, S.~Mikul{\'a}s, and Y.~Venema.
\newblock Axiomatizing complex algebras by games.
\newblock {\em Algebra Universalis}, 46:455--478, 2001.

\bibitem{igna93}
K.N. Ignatiev.
\newblock The provability logic of ${\Sigma}_1$-interpolability.
\newblock {\em Annals of Pure and Applied Logic}, 64:1--25, 1993.

\bibitem{Japa94}
G.K. Japaridze.
\newblock The logic of the arithmetical hiearchy.
\newblock {\em Annals of Pure and Applied Logic}, 66:89--112, 1994.

\bibitem{joo98}
J.J. Joosten.
\newblock Towards the interpretability logic of all reasonable arithmetical
  theories.
\newblock Master's thesis, University of Amsterdam, 1998.

\bibitem{joo:prol00}
J.J. Joosten and A.~Visser.
\newblock The interpretability logic of {\em all}\/ reasonable arithmetical
  theories.
\newblock {\em Erkenntnis}, 53(1--2):3--26, 2000.

\bibitem{jogo:mm08}
E. Goris.
\newblock Modal Matters for Interpretability Logic.
\newblock {\em Logic Journal of the Interest Group in Pure and Applied Logics}, 16: 371 - 412, August 2008.


\bibitem{Kent73}
C.F Kent.
\newblock The relation of {{\it A}} to {$Prov\ulcorner!A\urcorner$} in the
  {L}indenbaum sentence algebra.
\newblock {\em Journal of Symbolic Logic}, 38:359--367, 1973.

\bibitem{heyt:malo90}
P.P. Petkov, editor.
\newblock {\em Mathematical logic, {P}roceedings of the {H}eyting 1988 summer
  school in {V}arna, {B}ulgaria}. Plenum Press, Boston, 1990.

\bibitem{pudl85}
P.~Pudl{\'a}k.
\newblock Cuts, consistency statements and interpretations.
\newblock {\em Journal of Symbolic Logic}, 50:423--441, 1985.

\bibitem{shav:logi88}
V.~Shavrukov.
\newblock The logic of relative interpretability over {P}eano arithmetic (in
  {R}ussian).
\newblock Technical Report Report No.5, Steklov Mathematical Institute, Moscow,
  1988.

\bibitem{Sol76}
R.M. Solovay.
\newblock Provability interpretations of modal logic.
\newblock {\em Israel Journal of Mathematics}, 28:33--71, 1976.

\bibitem{tars:unde53}
A.~Tarski, A.~Mostowski, and R.~Robinson.
\newblock {\em Undecidable theories}.
\newblock North--Holland, Amsterdam, 1953.

\bibitem{Svej91}
V.~\u{S}vejdar.
\newblock Some independence results in interpretability logic.
\newblock {\em Studia Logica}, 50:29--38, 1991.

\bibitem{viss:prel88}
A.~Visser.
\newblock Preliminary notes on interpretability logic.
\newblock Technical Report LGPS 29, Department of Philosophy, Utrecht
  University, 1988.

\bibitem{viss:inte90}
A.~Visser.
\newblock Interpretability logic.
\newblock In {\em \cite{heyt:malo90}}, pages 175--209, 1990.

\bibitem{Vis91}
A.~Visser.
\newblock The formalization of interpretability.
\newblock {\em Studia Logica}, 50(1):81--106, 1991.

\bibitem{Vis93b}
A.~Visser.
\newblock A course on bimodal provability logic.
\newblock {\em Annals of Pure and Applied Logic}, pages 109--142, 1995.

\bibitem{vis97}
A.~Visser.
\newblock An overview of interpretability logic.
\newblock In M.~Kracht, M.~{de} Rijke, and H.~Wansing, editors, {\em Advances
  in modal logic '96}, pages 307--359. CSLI Publications, Stanford, CA, 1997.

\bibitem{viss:faith02}
A.~Visser.
\newblock Faith {\&} {F}alsity: a study of faithful interpretations and false
  ${\Sigma}^0_1$-sentences.
\newblock Logic Group Preprint Series 216, Department of Philosophy, Utrecht
  University, Heidelberglaan 8, 3584 CS Utrecht, October 2002.

\bibitem{Vuk96}
M~Vukovi\'c.
\newblock Some correspondence of principles in interpretability logic.
\newblock {\em Glasink Matemati\v{c}ki}, 31(51):193--200, 1996.

\end{thebibliography}

\end{document}